\documentclass[11pt,a4paper]{article}
\usepackage{fancyhdr}
\usepackage{amsmath}
\usepackage{amsthm}
\usepackage{amssymb}
\usepackage{wasysym}
\usepackage{paralist}
\usepackage{makeidx}
\usepackage[all]{xy}
\usepackage{mathdots}
\usepackage{yhmath}
\usepackage{mathtools}

\usepackage{young}
\usepackage[vcentermath]{youngtab}

\usepackage{graphicx}
\usepackage{epstopdf}

\input xy

\newcommand{\nc}{\newcommand}
 
 \nc{\cl}{\centerline}
 \renewcommand{\l}{{\rm len}}
 \nc{\SL}{{\rm SL}}
 \nc{\hatQ}{{\hat Q}}
 \nc{\sgn}{{\rm sgn}}
 
 \nc{\seee}{\mathbb C}
 
 \newcommand{\id}{{\rm id}}

 \nc{\hatlambda}{{\hat\lambda}}
 
 \nc{\daggerlambda}{{\lambda^\dagger}}

  \newcommand{\dotG}{{\dot{G}}}
    \newcommand{\dotL}{{\dot{L}}}

  \newcommand{\dotE}{{\dot{E}}}

  \newcommand{\resp}{{resp.\,}}

  \newcommand{\Hec}{{\rm Hec}}

\nc\diag{{\rm diag}}
\renewcommand{\vert}{{\,|\,}}
\nc{\hatL}{{\hat L}}

\nc{\barE}{{\bar   E}}
\nc{\D}{{\mathcal D}}
\nc{\E}{{\mathcal E}}
\nc{\F}{{\mathcal F}}
\nc{\FF}{{\mathcal F}}
\nc{\I}{{\mathcal I}}
\nc{\even}{{\rm e}}
\nc{\ep}{\epsilon}
\nc{\odd}{{\rm o}}
\nc{\Coker}{{\rm Coker}}
\nc{\olE}{{\overline E}}
\nc{\indBG}{{\rm ind}_B^G\,}
\nc{\indHG}{{\rm ind}_H^G\,}

\nc{\que}{{\mathbb Q}}
\nc{\barlambda}{{\bar\lambda}}
\nc{\barmu}{{\bar\mu}}
\nc{\barnu}{{\bar\nu}}
\nc{\bartau}{{\bar\tau}}
\nc{\barm}{{\bar m}}
\nc{\divind}{{\rm div.ind}}
\nc{\tl}{{\tilde{\lambda}}}
\nc{\dar}{\downarrow}

\nc{\Sym}{{\rm \Sym}}
\nc{\Symm}{{\rm Sym}}

\newcommand{\q}{\quad}
\newcommand{\de}{\delta}

\newcommand{\Mod}{{\rm Mod}}
\renewcommand{\mod}{{\rm mod}}

\newcommand{\Sp}{{\rm Sp}}

\newcommand{\bs}{\bigskip}

\renewcommand{\vert}{\,|\,}

\renewcommand{\sgn}{{\rm sgn}}

\newcommand{\pol}{{\rm pol}}

\xyoption{all}

\newcommand{\ind}{{\rm ind}}
\renewcommand{\vert}{\,|\,}
 \newcommand{\tbw}{\textstyle\bigwedge}
\newcommand{\zed}{{\mathbb Z}}
\newcommand{\Ext}{{\rm Ext}}

\newcommand{\Hom}{{\rm Hom}}
\newcommand{\cf}{{\rm cf}}

\renewcommand{\mod}{{\rm mod}}

\newcommand{\GL}{{\rm GL}}

\renewcommand{\mod}{{\rm{mod}}}

\nc{\geom}{{\rm geom}}
\nc{\rep}{{\rm rep}}

\newcommand{\ch}{{\rm ch\,}}

\newtheorem{definition}{Definition}[section]

\newtheorem{theorem}[definition]{Theorem}
\newtheorem{lemma}[definition]{Lemma}
\newtheorem{corollary}[definition]{Corollary}
\newtheorem{example}[definition]{Example}

\newtheorem{remark}[definition]{Remark}
\newtheorem{thm}{Theorem}[subsection]
\newtheorem{prop}[thm]{Proposition}
\newtheorem{lem}[thm]{Lemma}
\newtheorem{corol}[thm]{Corollary}
\newtheorem{exam}[thm]{Example}

\newtheorem{rmk}[thm]{Remark}

\begin{document}



\centerline{\bf Decompositions of some Specht modules I}

\bigskip

\centerline{Stephen Donkin  and Haralampos Geranios}

\bigskip

{\it Department of Mathematics, University of York, York YO10 5DD}\\

\medskip

{\tt stephen.donkin@york.ac.uk,  haralampos.geranios@york.ac.uk}

\bs

\bs

\centerline{11  November  2018}
\bs\bs\bs

\section*{Abstract}

\q We give a decomposition as a direct sum of indecomposable modules of several types of Specht modules in characteristic $2$.  These include the Specht modules labelled by hooks, whose decomposability was considered  by Murphy, \cite{Gwen}. Since the main arguments  are essentially no more difficult for Hecke algebras at parameter $q=-1$, we proceed in this level of generality.

\section{Introduction}

\q   Let $K$ be a field and $r$ a positive integer. We write $\Sigma_r$ for the symmetric group of degree $r$. For each partition  $\lambda$ of $r$  we have the Specht module $\Sp(\lambda)$ and for each composition $\alpha$ of $r$ the permutation module  $M(\alpha)$. The Specht module $\Sp(\lambda)$ may be viewed as a submodule of $M(\lambda)$.   James proved, \cite[13.17]{James},  that   unless the characteristic of $K$ is $2$ and $\lambda$ is $2$-singular, the space of homomorphisms $\Hom_{\Sigma_r}(\Sp(\lambda),M(\lambda)) $ is one dimensional. It follows that $\Sp(\lambda)$ has one dimensional endomorphism algebra and in particular that $\Sp(\lambda)$ is indecomposable (unless $K$ has characteristic $2$ and $\lambda$ is $2$-singular). 

\q We now suppose $K$ has characteristic $2$. Then, for $\lambda$ a $2$-singular partition, the Specht module $\Sp(\lambda)$ may certainly decompose but in general neither a criterion for decomposability nor the nature of a decomposition as a direct sum of indecomposable components, is known. The first example of such a module was discovered by James for the symmetric group $\Sigma_7$ and the Specht module $\Sp(5,1,1)$, see \cite{Jam1}. Some years later,  Murphy generalised James' example and in \cite{Gwen} she gave a criterion for the decomposability of Specht modules labelled by hook partitions, i.e. partitions of the form $\lambda=(a,1^b)$. More recently,  Dodge and  Fayers found in \cite{DF} some new decomposable Specht modules for partitions of the form $\lambda=(a,3,1^b)$. 

\q In the more general context of the Hecke algebras $H_q(r)$, Dipper and James showed in \cite{DJ2} that for $q\neq -1$ the corresponding Specht modules are indecomposable. Recently Speyer generalised Murphy's criterion regarding the decomposability of Specht modules labeled by hook partitions for Hecke algebras with $q=-1$, see \cite{Spey}.   

\q We here obtain many new families of decomposable Specht modules  for Hecke algebras at parameter $q=-1$ and describe explicitly their indecomposable components. More precisely, we give a decomposition of the Specht modules $\Sp(a,m-1,m-2,\ldots,2,1^b)$, with $a\geq m$, $b\geq 1$ and $a-m$ even and $b$ odd. Moreover, we show that the number of indecomposable summands in such decompositions is unbounded. We also point out that the decomposition of $\Sp(a,m-1,m-2,\ldots,2,1^b)$ lays the foundations for the discovery of many other families of decomposable Specht modules. In fact, using this approach we describe decompositions of Specht modules of the form $\Sp(a,3,1^b)$ which don't appear in the list produced by Dodge and Fayers. More results in this direction will appear in a follow up paper, \cite{DGMoreDeco}.

\q Our method is to compare the situation with an analogous problem for certain modules for the general linear groups and apply the Schur functor,  as in \cite[Section 2.1]{q_Schur}. 

\begin{remark} We shall produce a decomposition of the Specht module $\Sp(\lambda)$,   for certain partitions $\lambda$.  In these cases $\Sp(\lambda)$ will actually be a Young module (though not in general indecomposable). We note that in fact we already have a supply of cases in which $\Sp(\lambda)$ is a Young module, namely when $\lambda$ is,  in the terminology of \cite{DGInjective},  a Young partition, \cite[Proposition 3.2 (i)]{DGInjective}. 

\q However, in these cases  the $\Sp(\lambda)$ is the indecomposable Young module $Y(\lambda)$. So these partitions are not of interest from the point of view of our current investigation.
\end{remark}

\section{Preliminaries}

\subsection{Combinatorics}

\q The  standard  reference for  the polynomial representations of $\GL_n(K)$ is the monograph \cite{EGS}. Though we work in the quantised context this reference is appropriate  as the combinatorics is  essentially the same and we adopt the notation of \cite{EGS} wherever  convenient. Further details may also be found in the monograph \cite{q_Schur}, which treats the quantised case.

\quad We begin by introducing some of the associated combinatorics.  By a partition we mean an infinite  sequence $\lambda=(\lambda_1,\lambda_2,\ldots)$ of nonnegative integers with $\lambda_1\geq\lambda_2\geq \ldots$ and $\lambda_j=0$ for $j$ sufficiently large.   If $n$ is a positive integer such that $\lambda_j=0$ for $j>n$ we identify $\lambda$ with the finite sequence $(\lambda_1,\ldots,\lambda_n)$.  The length $\l(\lambda)$ of  a partition $\lambda=(\lambda_1,\lambda_2,\ldots)$ is $0$ if $\lambda=0$ and  is the positive integer $n$ such that  $\lambda_n\neq 0$, $\lambda_{n+1}=0$, if $\lambda\neq 0$. For a partition $\lambda$, we denote  by $\lambda'$ the transpose partition of $\lambda$. We define the degree of $\lambda=(\lambda_1,\lambda_2,\ldots)$ by $\deg(\lambda)=\lambda_1+\lambda_2+\cdots$.

\q We fix a positive integer $n$. We set $X(n)=\zed^n$. There is a natural partial order on $X(n)$. For $\lambda=(\lambda_1,\ldots,\lambda_n), \mu=(\mu_1,\ldots,\mu_n)\in X(n)$,  we write $\lambda\leq \mu$ if $\lambda_1+\cdots+\lambda_i\leq \mu_1+\cdots+\mu_i$ for $i=1,2,\ldots,n-1$ and $\lambda_1+\cdots+\lambda_n=\mu_1+\cdots+\mu_n$. We shall use the standard $\zed$-basis   $\ep_1,\ldots,\ep_n$ of   $X(n)$,  where $\ep_i=(0,\ldots,1,\ldots,0)$ (with $1$ in the $i$th position), for $1\leq i\leq n$. We write 
$\omega_i$ for the element $\ep_1+\cdots+\ep_i$ of $X(n)$, for $1\leq i\leq n$. 

\q We write $X^+(n)$ for the set of dominant $n$-tuples of integers, i.e., the set of elements $\lambda=(\lambda_1,\ldots,\lambda_n)\in X(n)$ such that $\lambda_1\geq \cdots\geq  \lambda_n$. We write $\Lambda(n)$ for the set of $n$-tuples of nonnegative integers and $\Lambda^+(n)$ for the set of partitions into at most $n$-parts, i.e.,  $\Lambda^+(n)=X^+(n)\bigcap \Lambda(n)$. We shall sometimes refer to elements of $\Lambda(n)$ as polynomial weights and to elements of $\Lambda^+(n)$ as polynomial dominant weights. For a nonnegative integer $r$ we write $\Lambda^+(n,r)$ for the set of partitions of $r$ into at most $n$ parts, i.e., the set of elements of $\Lambda^+(n)$ of degree $r$.

\q We fix a positive integer $l$.  We write $X_1(n)$ for the set of $l$-restricted partitions into at most $n$ parts, i.e., the set of elements $\lambda=(\lambda_1,\ldots,\lambda_n)\in \Lambda^+(n)$ such that $0\leq \lambda_1-\lambda_2,\ldots,\lambda_{n-1}-\lambda_n, \lambda_n<l$.  
 
 \q A dominant  weight $\lambda\in X^+(n)$ has a unique expression $\lambda=\lambda^0+l\barlambda$ with $\lambda^0\in X_1(n)$, $\barlambda\in X^+(n)$, moreover if $\lambda\in\Lambda^+(n)$ then $\barlambda\in \Lambda^+(n)$. We shall use this notation a great deal in what follows.

\subsection{Rational Modules and Polynomial Modules}

\q Let $K$ be a field. If $V,W$ are vector spaces over $K$, we write $V\otimes W$ for the tensor product $V\otimes_K W$.   We shall be working with the representation theory of quantum groups over $K$. By the category of quantum groups over $K$ we understand the opposite category of the category of Hopf algebras over $K$. Less formally we shall use the expression \lq\lq $G$ is a quantum group" to indicate that we have in mind a Hopf algebra over $K$ which we denote $K[G]$ and call the coordinate algebra of $G$.  We say that $\phi:G\to H$ is a morphism of quantum groups over $K$ to indicate that we have in mind a morphism of Hopf algebras over $K$, from $K[H]$ to $K[G]$, denoted $\phi^\sharp$ and called the co-morphism of $\phi$.   We will say $H$ is a quantum subgroup of the quantum group $G$, over $K$, to indicate that $H$ is a quantum group with coordinate algebra $K[H]=K[G]/I$, for some Hopf ideal $I$ of $K[G]$, which we call the defining ideal of $H$.  The inclusion morphism $i:H\to G$ is the morphism of quantum groups whose co-morphism $i^\sharp:K[G]\to K[H]=K[G]/I$ is the natural map. 

\q Let $G$ be a quantum group over $K$. The category of  left (\resp right) $G$-modules is the the category of right (\resp left) $K[G]$-comodules.  We write $\Mod(G)$ for the category of left $G$-modules and $\mod(G)$ for the category of finite dimensional left $G$-modules.  We shall also call a $G$-module a rational $G$-module (by analogy with the representation theory of algebraic groups).  A $G$-module will mean a left $G$-module unless  indicated otherwise. For a finite dimensional $G$-module $V$ and a non-negative integer $d$ we write $V^{\otimes d}$ for the $d$-fold tensor product $V\otimes\cdots \otimes V$ and we write $V^{(d)}$ for the direct sum $V\oplus\cdots\oplus V$ of $d$ copies of $V$.

 \q Let $V$ be a finite dimensional $G$-module with structure map $\tau: V\to V\otimes K[G]$.  The coefficient space $\cf(V)$ of $V$ is the subspace of $K[G]$ spanned by the \lq\lq coefficient elements" $f_{ij}$, $1\leq i,j\leq m$, defined with respect to a basis $v_1,\ldots, v_m$ of $V$, by the equations 
 $$\tau(v_i)=\sum_{j=1}^m v_j\otimes f_{ji}$$
 for $1\leq i\leq m$.  The coefficient space $\cf(V)$ is independent of the choice of basis and is a subcoalgebra of $K[G]$.

\q We fix a positive integer $n$. We shall be working with $G(n)$, the quantum general linear group of degree $n$, as in \cite{q_Schur}.   We fix a non-zero element $q$ of $K$. 
We have a $K$-bialgebra $A(n)$ given by generators $c_{ij}$, $1\leq i,j\leq n$, subject to certain relations (depending on $q$) as in \cite[0.20]{q_Schur}.  The comultiplication map  $\delta:A(n)\to A(n)\otimes A(n)$ satisfies $\de(c_{ij})=\sum_{r=1}^n c_{ir}\otimes c_{rj}$ and the augmentation map $\ep:A(n)\to K$ satisfies $\ep(c_{ij})=\de_{ij}$ (the Kronecker delta), for $1\leq i,j\leq n$.  The elements $c_{ij}$ will be called the coordinate elements and we define the determinant element
$$d_n=\sum_{\pi\in \Symm(n)} \sgn(\pi) c_{1,\pi(1)}\ldots c_{n,\pi(n)}.$$
Here  $\sgn(\pi)$ denotes the sign of the permutation $\pi$. We form the Ore localisation $A(n)_{d_n}$. The comultiplication map $A(n)\to A(n)\otimes A(n)$ and augmentation map $A(n)\to K$ extend uniquely to $K$-algebraic maps $A(n)_{d_n}\to A(n)_{d_n}\otimes A(n)_{d_n}$ and $A(n)_{d_n}\to K$, giving $A(n)_{d_n}$ the structure of a Hopf algebra. By the quantum general linear group $G(n)$ we mean the quantum group over $K$ with coordinate algebra $K[G(n)]=A(n)_{d_n}$.

\q We write $T(n)$ for the quantum subgroup of $G(n)$ with defining ideal generated by all $c_{ij}$ with $1\leq i,j\leq n$, $i\neq j$.  We write $B(n)$ for quantum subgroup of $G(n)$ with defining ideal generated by all $c_{ij}$ with $1\leq i <j\leq n$. We call $T(n)$ a maximal torus and $B(n)$ a Borel subgroup of $G(n)$ (by analogy with the classical case).

\q We now assign to a finite dimension rational $T(n)$-module its formal character. We form the integral group ring $\zed X(n)$. This has $\zed$-basis of formal exponentials $e^\lambda$, which multiply according to the rule $e^\lambda e^\mu=e^{\lambda+\mu}$, $\lambda,\mu\in X(n)$. For $1\leq i\leq n$ we define ${\bar c}_{ii}=c_{ii}+I_{T(n)}\in K[T(n)]$, where $I_{T(n)}$ is the defining ideal of the quantum subgroup $T(n)$ of $G(n)$.  Note that ${\bar c}_{11}\ldots {\bar c}_{nn}=d_n+I_{T(n)}$, in particular each ${\bar c}_{ii}$ is invertible in $K[T(n)]$. For $\lambda=(\lambda_1,\ldots,\lambda_n)\in X(n)$ we define ${\bar c}^\lambda={\bar c}_{11}^{\lambda_1}\ldots {\bar c}_{nn}^{\lambda_n}$. The elements ${\bar c}^\lambda$, $\lambda\in X(n)$, are group-like and form a $K$-basis of $K[T(n)]$. 
For $\lambda=(\lambda_1,\ldots,\lambda_n)\in X(n)$, we write $K_\lambda$ for $K$ regarded as a (one dimensional) $T(n)$-module with structure map $\tau:K_\lambda\to K_\lambda\otimes K[T(n)]$ given by $\tau(v)=v\otimes {\bar c}^\lambda$, $v\in K_\lambda$.  For a finite dimensional  rational $T(n)$-module $V$ with structure map $\tau:V\to V\otimes K[T(n)]$  and $\lambda\in X(n)$ we have the weight space 
$$V^\lambda=\{v\in V\vert \tau(v) =v\otimes {\bar c}^\lambda\}.$$  
Moreover, we have the weight space decomposition $V=\bigoplus_{\lambda\in X(n)} V^\lambda$. 
We say that $\lambda\in X(n)$ is a weight of $V$ if $V^\lambda\neq 0$. 
The dimension of a finite dimensional vector space $V$ over $K$ will be denoted by $\dim V$. 
The character $\ch V$ of a finite dimensional rational $T(n)$-module $V$ is the element of $\zed X(n)$ defined by
$\ch V=\sum_{\lambda\in X(n)} \dim V^\lambda e^\lambda$.   

\q For each $\lambda\in X^+(n)$ there is an irreducible rational $G(n)$-module $L_n(\lambda)$ which has unique highest weight $\lambda$, and $\lambda$ occurs as a weight with multiplicity one. The modules $L_n(\lambda)$, $\lambda\in X^+(n)$, form a complete set of pairwise non-isomorphic irreducible rational  $G$-modules. For a finite dimensional rational $G(n)$-module $V$ and $\lambda\in X^+(n)$ we write $[V:L_n(\lambda)]$ for the multiplicity of $L_n(\lambda)$ as a composition factor of $V$. 

\q We write $D_n$ for the one dimensional $G(n)$-module corresponding to the determinant. Thus $D_n$ has structure map $\tau:D_n\to D_n\otimes K[G]$, given by $\tau(v)=v\otimes d_n$, for $v\in D_n$. We have  $D_n=L_n(\omega_n)=L_n(1,1,\ldots,1)$. We write $E_n$ for the natural $G(n)$-module. Thus  $E_n$ has basis $e_1,\ldots,e_n$,  and the structure map $\tau:E_n\to   E_n\otimes K[G(n)]$ is given by $\tau(e_i)=\sum_{j=1}^n e_j\otimes c_{ji}$. We also have that $E_n=L_n(1,0,\dots,0)$.

\q A finite dimensional $G(n)$-module $V$ is called polynomial if $\cf(V) \leq A(n)$. The modules $L_n(\lambda)$, $\lambda\in \Lambda^+(n)$, form a complete set of pairwise non-isomorphic irreducible polynomial $G(n)$-modules. We have a grading   $A(n)=\bigoplus_{r=0}^\infty  A(n,r)$ in  such a way that each $c_{ij}$ has degree $1$. Moreover each $A(n,r)$ is a finite dimensional subcoalgebra of $A(n)$. The dual algebra $S(n,r)$ is known as the Schur algebra.  A finite dimensional $G(n)$-module $V$ is polynomial of degree $r$ if $\cf(V)\leq A(n,r)$.  We write $\pol(n)$ (\resp $\pol(n,r)$)  for the full subcategory of $\mod(G(n))$ whose objects are the polynomial modules (\resp the modules which are polynomial of degree $r$).

\q For an arbitrary finite dimensional polynomial $G(n)$-module we may write $V$ uniquely as a direct sum $V=\bigoplus_{r=0}^\infty V(r)$ in such a way that $V(r)$ is polynomial of degree $r$, for $r\geq 0$.  Let $r\geq 0$. The modules $L_n(\lambda)$, $\lambda\in\Lambda^+(n,r)$, form a complete set of pairwise non-isomorphic irreducible polynomial $G(n)$-modules which are polynomial of degree $r$.  We write $\mod(S)$ for the category of left modules for a finite dimensional $K$-algebra $S$.  The category 
 $\pol(n,r)$ is naturally equivalent to the category $\mod(S(n,r))$.

\q We shall also need modules induced from $B(n)$ to $G(n)$.  (For details of the induction functor $\Mod(B(n))\to \Mod(G(n))$ see, for example, \cite{DStd}.)  For $\lambda\in X(n)$ there is a unique (up to isomorphism) one dimensional $B(n)$-module whose restriction to $T(n)$ is  $K_\lambda$. We also denote this module by $K_\lambda$. The induced module $\ind_{B(n)}^{G(n)} K_\lambda$ is non-zero if and only if $\lambda\in X^+(n)$. For $\lambda\in X^+(n)$ we set $\nabla_n(\lambda)=\ind_{B(n)}^{G(n)}  K_\lambda$. Then $\nabla_n(\lambda)$ is finite dimensional  and its character is given by Weyl's character formula. In fact, for $\lambda\in \Lambda^+(n)$ the character of $\nabla_n(\lambda)$ is the Schur symmetric function corresponding to $\lambda$. The $G(n)$-module socle of $\nabla_n(\lambda)$ is $L_n(\lambda)$. The module $\nabla_n(\lambda)$ has unique highest weight $\lambda$ and this weight occurs with multiplicity one.

\q A filtration $0=V_0\leq V_1\leq \cdots\leq V_r=V$ of  a finite dimensional rational $G(n)$-module $V$ is said to be {\it good} if for each $1\leq i\leq r$ the quotient $V_i/V_{i-1}$ is either zero or isomorphic to $\nabla_n(\lambda^i)$ for some $\lambda^i\in X^+(n)$.  For a rational $G(n)$-module $V$ admitting a good filtration for each $\lambda\in X^+(n)$, the multiplicity 
$|\{1\leq i\leq r\vert V_i/V_{i-1}\cong \nabla_n(\lambda)\}|$ is independent of the choice of the good filtration, and will be denoted $(V:\nabla_n(\lambda))$. 

\q For $\lambda,\mu\in X^+(n)$ we have $\Ext^1_{G(n)}(\nabla_n(\lambda),\nabla_n(\mu))=0$ unless $\lambda>\mu$. Given Kempf's Vanishing Theorem, \cite[Theorem 3.4]{q_Schur},  this follows exactly as in the classical case, e.g., \cite[Lemma 3.2.1]{D4}, (or the original source \cite[Corollary (3.2)]{CPS}). It  follows that if $V$ has a good  filtration 
$0=V_0\leq V_1\leq \cdots\leq V_t =V$ with sections $V_i/V_{i-1}\cong \nabla_n(\lambda_i)$, $1\leq i\leq t$, and $\mu_1,\ldots,\mu_t$ is a reordering of the $\lambda_1,\ldots,\lambda_t$ such that $\mu_i<\mu_j$ implies that $i<j$ then there is a good filtration $0=V_0'<V_1'<\cdots <V_t'=V$ with $V_i'/V_{i-1}'\cong \nabla_n(\mu_i)$, for $1\leq i\leq t$. 

\q Similarly it will be of great practical use  to know that\\
 $\Ext^1_{G(n)}(\nabla_n(\lambda),\nabla_n(\mu))=0$ when $\lambda$ and $\mu$ belong to different blocks. Here  the relationship with cores of partitions diagrams will be crucial for us.  For a partition $\lambda$ we denote by $[\lambda]$ the corresponding partition diagram (as in \cite{EGS}). The $l$-core of $[\lambda]$ is the diagram obtained by removing skew $l$-hooks,  as in \cite{James}. If $\lambda,\mu\in \Lambda^+(n,r)$ and $[\lambda]$ and $[\mu]$ have different $l$-cores then the simple modules $L_n(\lambda)$ and $L_n(\mu)$ belong to different blocks and it follows in particular that $\Ext^i_{S(n,r)}(\nabla(\lambda),\nabla(\mu))=0$, for all $i\geq 0$.  A precise  description of the blocks of the $q$-Schur algebras was found by Cox, see \cite[Theorem 5.3]{Cox}. For a polynomial $G(n)$-module $V$ and an $l$-core $\gamma$ we mean by the expression, the component of $V$ corresponding to $\gamma$, the sum of all block components for blocks consisting of partitions with core $\gamma$.

\q  For $\lambda\in \Lambda^+(n,r)$ we write $I_n(\lambda)$ for the injective envelope of $L_n(\lambda)$ in the category of polynomial modules. We have that $I_n(\lambda)$ is a finite dimensional module which is polynomial of degree $r$. Moreover, the module $I_n(\lambda)$ has a good filtration and we have the reciprocity formula $(I_n(\lambda):\nabla_n(\mu))=[\nabla_n(\mu):L_n(\lambda)]$ see e.g., \cite[Section 4, (6)]{DStd}.

\subsection{The Frobenius Morphism}

\q It will be important for us to make a comparison with the classical case $q=1$.   In this case we will write $\dotG(n)$ for  $G(n)$ and write $x_{ij}$ for the coordinate element $c_{ij}$, $1\leq i,j\leq n$. Also, we write $\dotL_n(\lambda)$ for the $\dotG$-module $L_n(\lambda)$, $\lambda\in X^+(n)$, and write $\dotE_n$ for $E_n$.

\q  We return to the general situation.  If $q$ is not a root or unity, or if $K$ has characteristic $0$ and $q=1$ then all $G(n)$-modules are completely reducible, see e.g.,  \cite[Section 4, (8)]{DStd}.  We therefore assume from now on that $q$ is a root of unity and that if $K$ has characteristic $0$ then $q\neq 1$.  Also, from now on, $l$ is the smallest positive integer such that $1+q+\cdots+q^{l-1}=0$.

\q Now we have a morphism of Hopf algebras $\theta: K[\dotG(n)]\to K[G(n)]$ given by $\theta(x_{ij})=c_{ij}^l$, for $1\leq i,j\leq n$.  We write $F:G(n)\to \dotG(n)$ for the morphism of quantum groups such that $F^\sharp=\theta$.  Given a $\dotG(n)$-module $V$ we write $V^F$ for the corresponding $G(n)$-module. Thus, $V^F$ as a vector space is $V$ and if the $\dotG(n)$-module $V$ has structure map $\tau:V\to V\otimes K[\dotG(n)]$ then $V^F$ has structure map $(\id_V \otimes F)\circ \tau: V^F \to  V^F\otimes K[G(n)]$, where $\id_V:V\to V$ is the identity map on the vector space  $V$.  

\q  For an element $\phi=\sum_{\xi\in X(n)} a_\xi e^\xi$ of $\zed X(n)$ we write $\phi^F$ for the element $\sum_{\xi\in X(n)}a_\xi e^{l\xi}$.  Then, for a finite dimensional $\dotG(n)$-module $V$ we have $\ch V^F= (\ch V)^F$.   Moreover, we have the following relationship between the irreducible modules for $G(n)$ and ${\dotG(n)}$, see \cite[Section 3.2, (5)]{q_Schur}.

\bs

{\bf Steinberg's Tensor Product Theorem}  \sl For  $\lambda^0\in X_1(n)$ and $\barlambda\in X^+(n)$ we have 
$$L_n(\lambda^0+l\barlambda)\cong L_n(\lambda^0)\otimes \dotL_n(\barlambda)^F.$$

\rm

\bs

\q  Usually we shall abbreviate the quantum groups  $G(n)$, $B(n)$, $T(n)$ to  $G$, $B$, $T$ and $\dotG(n)$ to $\dotG$.   Likewise,  we usually abbreviate the modules 
$L_n(\lambda)$, $\nabla_n(\lambda)$,  $I_n(\lambda)$ and $\dotL_n(\lambda)$ to $L(\lambda)$, $\nabla(\lambda)$, $I(\lambda)$ and $\dotL(\lambda)$, for $\lambda\in \Lambda^+(n)$, and abbreviate the modules $E_n$ and $D_n$ to $E$ and $D$.

\subsection{A truncation functor}

\q Let $N,n$ be positive integers with $N\geq n$.   We identify $G(n)$ with the quantum subgroup of $G(N)$ whose defining ideal is generated by all $c_{ii}-1$, $n<i\leq N$, and all $c_{ij}$ with $1\leq i\neq j\leq N$ and $i>n$ or $j>n$.  We have an exact functor (the truncation functor) $d_{N,n}:\pol(N)\to \pol(n)$ taking $V\in \pol(N)$ to the $G(n)$ submodule  $\bigoplus_{\alpha\in \Lambda(n)}  V^\alpha$  of $V$ and taking a morphism of polynomial modules $V\to V'$ to its restriction $d_{N,n}(V)\to d_{N,n}(V')$. For a discussion of this functor at the level of modules for Schur algebras in the classical case see \cite[Section 6.5]{EGS}.

 \q By \cite[Section 4.2]{q_Schur}, the functor $d_{N,n}$ has the following properties:

(i) for polynomial $G(N)$-modules $X,Y$ we have $d_{N,n}(X\otimes Y)=d_{N,n}(X)\otimes d_{N,n}(Y)$;\\
(ii)  for $\lambda\in \Lambda^+(N,r)$ and $X_\lambda=L_N(\lambda)$ or $\nabla_N(\lambda)$  then $d_{N,n}(X_\lambda)\neq0$ if and only if $\lambda\in \Lambda^+(n,r)$;\\
(iii)  for  $\lambda\in \Lambda^+(n,r)$, $d_{N,n}(L_N(\lambda))=L_n(\lambda)$ and $d_{N,n}(\nabla_N(\lambda))=\nabla_n(\lambda)$.

\medskip

\q Let $\lambda\in \Lambda^+(N,r)$,  for some $r\geq 0$, $\alpha\in \Lambda(N,r)$ and $\lambda_i=\alpha_i=0$, for $n<i\leq N$. We identify $\lambda$ and $\alpha$ with elements of $\Lambda^+(n,r)$ and $\Lambda(n,r)$ in the obvious way.  It follows that $\dim L_N(\lambda)^\alpha=\dim L_n(\lambda)^\alpha$.

\subsection{Connections with the Hecke algebras}

\q We now record some connections with representations of  Hecke algebra of type $A$. We fix a positive integer $r$.  We write $\l(\pi)$ for the length of a permutation $\pi$.   The Hecke algebra $\Hec(r)$ is the $K$-algebra with basis $T_w$, $w\in \Symm(r)$, and multiplication satisfying

\begin{align*}&T_wT_{w'}=T_{ww'}, \quad \hbox{ if } \l(ww')=\l(w)+\l(w'), \hbox{and}\cr
&(T_s+1)(T_s-q)=0
\end{align*}
for $w,w'\in \Symm(r)$ and a basic transposition $s\in \Symm(r)$. For brevity we will denote the Hecke algebra $\Hec(r)$ by $H(r)$.

\q  For $\lambda$ a partition of degree $r$ we denote  by $\Sp(\lambda)$ the corresponding (Dipper-James) Specht module and by $Y(\lambda)$ the corresponding Young module. For $\alpha\in \Lambda(n,r)$ we write $M(\alpha)$ for the permutation module corresponding to $\alpha$.

\q Let $n\geq r$. We have the Schur functor $f:\mod(S(n,r))\to \mod(H(r))$, see \cite[Section 2.1]{q_Schur}. By \cite[Sections 4.4 and 4.5]{q_Schur} we have that the functor $f$ has the following properties:

(i) $f$ is exact;\\
(ii) for $\lambda\in \Lambda^+(n,r)$ we have $f(\nabla(\lambda))=\Sp(\lambda)$;\\
(iii)  for  $\lambda\in \Lambda^+(n,r)$ we have  $f(I(\lambda))=Y(\lambda)$.\\

\q For a finite string of non-negative integers $\alpha=(\alpha_1,\ldots,\alpha_m)$ we  have the polynomial $G(n)$-modules 
$$S^\alpha E=S^{\alpha_1}E\otimes\cdots \otimes S^{\alpha_m}E$$
 and 
 $$\tbw^\alpha E =\tbw^{\alpha_1} E \otimes \cdots \otimes \tbw^{\alpha_m} E.$$
 
 \q For  $\alpha\in\Lambda(n,r)$ we write $H(\alpha)$ for the subalgebra $H(\alpha_1)\otimes\cdots \otimes H(\alpha_n)$ of $H(r)$.  By \cite[Section 2.1, (20)]{q_Schur} we have that: 

i) $f(S^\alpha E)=H(r)\otimes_{H(\alpha)}K=M(\alpha)$;\\
ii) $f(\Lambda^\alpha E)=H(r)\otimes_{H(\alpha)}K_{\sgn}$.

\section{Some weight space multiplicities}

\q Recall that our standard assumption is that $q$ is a root of unity and that if the base filed $K$ has characteristic $0$ then $q\neq 1$.  Also, $l$ is the smallest positive integer such that $1+q+\cdots+q^{l-1}=0$.

\q If $K$ has characteristic $p>0$, the base $p$-expansion of a positive integer $r$ will be written $r=\sum_{i=0}^\infty p^ir_i$ or just $r=\sum_{i=0}^N p^i r_i$, if $r<p^{N+1}$.

\begin{definition} Let $r, b$ be integers  with $r\geq 0$.  We shall say that  the pair $(r,b)$ is $p$-special if 

 i) for $p=0$ we have that $-r\leq b\leq r$ and $r-b$ is even;

ii) for $p>0$ and $r=\sum_{i=0}^\infty p^i r_i$ be the $p$ expansion of $r$, there exists an expression $b=\sum_{i=0}^\infty p^i t_i$ with $-r_i\leq  t_i \leq r_i$ and $r_i-t_i$  even for all $i\geq 0$.
 
\end{definition}

\begin{definition}

Let $s, a$ be integers  with $s\geq 0$.  We write $s=s_0+l\bar{s}$ with $0\leq s_0\leq l-1$. We shall say that  the pair $(s,a)$ is $(l,p)$-special if  there exists an expression $a=a_0+l\bar{a}$, with $-s_0\leq  a_0 \leq s_0$, $s_0-a_0$ even and $(\bar{s},\bar{a})$ $p$-special.

\end{definition}

\q From Steinberg's tensor product theorem for $G(2)$,  we get  the following, which  in the case $l=2$  is crucial for our analysis of a decomposition of some Specht modules. 

\begin{lemma} For a partition $(c,d)\in \Lambda^+(2,r)$ and for  $(a,b)\in \Lambda(2,r)$, we have 
$$\dim L(c,d)^{(a,b)}=\begin{cases} 1, & \hbox{   if  }    (c-d,a-b) \hbox{ is $(l,p)$-special};\cr
0, & \hbox{ otherwise.}
\end{cases}
$$

\end{lemma}

\smallskip

\q For a positive integer $m$ we write $\delta_m$ for the partition $(m,m-1,\ldots,1)$ (of length $m$)  and $\sigma_m$ for $(l-1)\de_m$. 

\begin{remark}
   Recall that if $\lambda\in \Lambda^+(n,r)$, $\alpha\in \Lambda(n,r)$ and $L(\lambda)^\alpha \neq 0$ then $\alpha\leq \lambda$. This implies that the length of $\alpha$ is at least the length of $\lambda$. If $\lambda$ and $\alpha$ have length $m$ then  
   $$\dim L(\lambda)^\alpha=\dim L_m(\lambda)^\alpha=\dim (D_m\otimes L(\lambda-\omega_m))^\alpha=\dim L(\lambda-\omega_m)^{\alpha-\omega_m}.$$

\end{remark}

\begin{lemma} For $m>0$, let  $a\geq m(l-1)$ and $b\geq (m-1)(l-1)$. Let $n\geq m$  and $\mu\in \Lambda^+(n)$.  

(i)  If  $\mu_{m+1}=0$  and $m>2$ then 
\begin{align*}\dim L(\sigma_m&+l\mu)^{(a,b,(m-2)(l-1),\ldots,2(l-1),(l-1))}\cr
&=\dim L(\sigma_{m-1}+l\mu)^{(a-(l-1),b-(l-1),(m-3)(l-1),\ldots,l-1)}
\end{align*}

and in fact  $L(\sigma_m+l\mu)^{(a,b,(m-2)(l-1),\ldots,2(l-1),(l-1))}=0$ unless $\mu_i=0$ for all $i>2$.

(ii) We have 
\begin{align*}\dim L(\sigma_m&+l\mu)^{(a,b,(m-2)(l-1),\ldots,2(l-1),(l-1))}\cr
&=\dim L(\sigma_1+l\mu)^{(a-(m-1)(l-1),b-(m-1)(l-1))}.
\end{align*}

(iii) If $a+m\equiv 0$ and $b+m\equiv 1\ \mod\ l$  and $\mu=(c,d)$ then  

\begin{align*}\dim L(\sigma_m&+l\mu)^{(a,b,(m-2)(l-1),\ldots,2(l-1),(l-1))}\cr
&=\begin{cases}1, & \hbox{ if }  (c-d,u-v) \hbox{ is $p$-special};\cr
0, & \hbox{otherwise}.
\end{cases}
\end{align*}
where  $a-m(l-1)=lu$ and $b-(m-1)(l-1)=lv$.

\end{lemma}

\begin{proof}   (i) If $\mu_{m+1}=0$ then  $\sigma_m+l\mu$ and

 $(a,b,(m-2)(l-1),\ldots,2(l-1),(l-1))$ are partitions of length $m$ with final term at least $l-1$. Hence, from Remark 3.4,  we obtain 
\begin{align*}
 \dim L&(\sigma_m+l\mu)^{(a,b,(m-2)(l-1),\ldots,2(l-1),(l-1))}\cr
 &=\dim L(\sigma_m+l\mu-(l-1)\omega_m)^{(a,b,(m-2)(l-1),\ldots,2(l-1),(l-1))-(l-1)\omega_m}\cr
 &=\dim  L(\sigma_{m-1}+l\mu)^{(a-(l-1),b-(l-1),(m-3)(l-1),\ldots,2(l-1),(l-1))}.
\end{align*}

This proves the first assertion. 

Now if  $L(\sigma_m+l\mu)^{(a,b,(m-2)(l-1),\ldots,2(l-1),(l-1))}\neq 0$ then 
$$\sigma_m+l\mu\geq (a,b,(m-2)(l-1),\ldots,2(l-1),(l-1))$$
which gives $\mu_{m+1}=0$. The second point follows now by induction on $m$.

(ii) This follows by  repeated application of (i).

(iii) We have $L(\sigma_1+l\mu)=L(\sigma_1)\otimes \dotL(\mu)^{F}$. We consider $L(\sigma_1+l\mu)$ as a $G(2)$-module. A weight of $L(\sigma_1)\otimes \dotL(\mu)^{F}$ has the form  $\alpha+l\beta$, where 
$$\alpha\in \{(l-1,0),(l-2,1),\ldots,(0,l-1)\}$$
 and $\beta$ is a weight of $\dotL(\mu)$  and all non-zero weight spaces are one dimensional.  Since $a-(m-1)(l-1)$ is congruent to $l-1$ (modulo $l$) the only solution for $\alpha$ is $(l-1,0)$ and the  dimension of the weight space 
 $$L(\sigma_1+l\mu)^{(a-(m-1)(l-1),b-(m-1)(l-1))}$$
  is equal to the number of weights $\beta$ of $\dotL(\mu)$ such that 
 $$(l-1,0)+l\beta=(a-(m-1)(l-1),b-(m-1)(l-1))$$
 so this dimension is $1$ if  $(a-m(l-1),b-(m-1)(l-1))$ is a weight of $\dotL(\mu)^{F}$, i.e., if $(u,v)$ is a weight of $\dotL(\mu)$, and $0$ otherwise.   Now the result follows from Lemma 3.3 (applied in the classical case $q=1$).

\end{proof}

\section{Decompositions of some polynomial modules}

\q The main technical point of this paper is an analysis, for $l=2$,  of a certain block component of a  module  of the form 
$$S^aE\otimes \tbw^bE   \otimes   \tbw^{m-2}E\otimes \cdots \otimes \tbw^2E\otimes E.$$

\q Suppose that $n\geq r$. Then, for $\alpha\in  \Lambda(n,r)$, the module $S^\alpha E$ is injective and we have 

$$S^\alpha E=\bigoplus_{\lambda\in \Lambda^+(n,r)} I(\lambda)^{({d_{\lambda\alpha}})}\eqno{(1)}$$
where $d_{\lambda\alpha}=\dim L(\lambda)^\alpha$, \cite[Section 2.1 (8)]{q_Schur}.

\q  Applying the Schur functor to (1), for any $\alpha\in \Lambda(n,r)$,  we get 
$$M(\alpha)=\bigoplus_{\lambda\in \Lambda^+(n,r)} Y(\lambda)^{(d_{\lambda\alpha})}\eqno{(2)}$$

 \q From Lemma 3.5 and equations (1) and (2) above we get the following two corollaries.

\begin{corollary} For $m\geq 2$, let $a\geq m(l-1)$ and $b\geq (m-1)(l-1)$.   Suppose that $a+m\equiv 0$ and $b+m\equiv 1\ \mod\ l$.    The  component of 
$$S^a E\otimes S^b E \otimes S^{(m-2)(l-1)} E\otimes\cdots \otimes S^{2(l-1)} E \otimes S^{l-1}E$$
 corresponding to the core $\sigma_m$ is 
$$\bigoplus_{\mu} I(\sigma_m+l\mu)$$
where the sum is over all partitions $\mu=(c,d)$ such that, $c+d=u+v$ and $(c-d,u-v)$ is $p$-special, where $a-m(l-1)=lu$ and $b-(m-1)(l-1)=lv$.
\end{corollary}

 \q Applying the Schur functor we then obtain the following.

\begin{corollary}  For $m\geq 2$, let $a\geq m(l-1)$ and $b\geq (m-1)(l-1)$.   Suppose that $a+m\equiv 0$ and $b+m\equiv 1 \ \mod\ l$.    The  block component 
$$M(a,b,(m-2)(l-1),\ldots,2(l-1),l-1)$$
 corresponding to the core $\sigma_m$ is 
$$\bigoplus_{\mu} Y(\sigma_m+l\mu)$$
where the sum is over all partitions $\mu=(c,d)$ such that, $c+d=u+v$ and $(c-d,u-v)$ is $p$-special, where $a-m(l-1)=lu$ and $b-(m-1)(l-1)=lv$.

\end{corollary}

\section{Adapted  partitions and symmetric polynomials}

\q For $\lambda\in \Lambda^+(n)$ we write $s(\lambda)$ for the Schur symmetric function corresponding to $\lambda$. The elements $s(\lambda)$ of $\zed X(n)$, as $\lambda$ varies over partitions with at most $n$ parts, form a $\zed$-basis of the ring of symmetric functions $\zed[x_1,\ldots,x_n]^{\Sigma_n}$ .

\q Let $\gamma\in \Lambda^+(n)$ be an $l$-core. For a polynomial $G(n)$-module $V$ we write $V(\gamma)$ for the component of $V$ corresponding to $\gamma$.   For $\lambda\in \Lambda^+(n)$ we define  $C^*_\gamma(s(\lambda))$ to be $s(\lambda)$ if the core of $\lambda$ is $\gamma$ and $0$ otherwise. We extend $C^*_\gamma$ additively to an endomorphism of the ring of symmetric functions $\zed[x_1,\ldots,x_n]^{\Sigma_n}$.

\begin{lemma} Let $\gamma\in \Lambda^+(n)$ be an $l$-core. For a finite dimensional polynomial module $V$ we have
$$\ch V(\gamma)=C^*_\gamma(\ch V).$$

\end{lemma}

\begin{proof} Since both sides are additive on short exact sequences of $G(n)$-modules, it is enough to check for a set of polynomial modules that generate the Grothendieck group of finite dimensional polynomial 
$G(n)$-modules. Hence it is enough to check for $V=\nabla(\lambda)$, $\lambda\in \Lambda^+(n)$, and for these modules it is clear from the definition.

\end{proof}

\q From now on we restrict attention to the case $l=2$. We need to keep  track of the part of a symmetric polynomial  corresponding to the core $\sigma_m=(m,m-1,.\ldots,1)$, $m\geq 0$ (where $\sigma_0=0$). To this end we introduce the following notion.

\begin{definition}  Let $m$ be a non-negative integer and $\lambda=(\lambda_1,\lambda_2,\ldots)$ be a partition. We say that $\lambda$ is $m$-adapted if $\lambda_i> m-i$, for all $i\geq 1$ with $\lambda_i>0$.
\end{definition}

\q We write $C_m$  for the additive endomorphism of the ring of symmetric functions in $n$ variables such that
$$C_m(s(\lambda))=\begin{cases} s(\lambda), & \hbox{ if }  \lambda \hbox{ is $m$-adapted};\cr
0, & \hbox{ otherwise.}
\end{cases}$$

\begin{lemma}   For a symmetric polynomial  $g$ (in $n$ variables) and $0\leq r\leq n$  we have
$$C_{m+1}(gs(1^r))=C_{m+1}(C_m(g)s(1^r)).$$
\end{lemma}

\begin{proof} It is enough to check this for $g=s(\lambda)$ for a partition $\lambda$, with at most $n$ parts.  By Pieri's Formula, \cite[5.17]{Mac},  we have 
$$s(\lambda) s(1^r)=\sum_\mu s(\mu)$$
where the sum is over all partitions $\mu$ with at most $n$ parts, whose Young diagram may be obtained from the Young diagram of $\lambda$ by adding a box in each of $r$ distinct rows.
 Hence we have $\mu_i\leq \lambda_i+1$ for each $\mu$ appearing in the above sum. 

\q If $\lambda$ is not $m$ adapted then $C_m(s(\lambda))=0$. Moreover, in this case, we have  $\lambda_i\leq m-i$ for some $i$, and so, for $\mu$ appearing in the above sum we have $\mu_i\leq \lambda_i+1\leq (m+1)-i$. Hence, $\mu$ is not $(m+1)$-adapted and $C_{m+1}(s(\mu))=0$. Therefore we get  $C_{m+1}(s(\lambda) s(1^r))=0$.

\q Suppose now that $\lambda$ is $m$-adapted. Thus we have 
$$C_{m+1}(C_m(s(\lambda))s(1^r))=C_{m+1}(s(\lambda) s(1^r))$$
 and we are done.

\end{proof}

\begin{lemma}   Let $m\geq 2$, $a\geq m$ and $b\geq m-1$. Then we have, for all $n$ sufficiently large, 
\begin{align*}C_m&(s(a)s(1^b) s(1^{m-2})\ldots s(1))\cr&
=
s(a+1,m-1,\ldots,2,1^{b-m+1})+s(a,m-1,\ldots,2,1^{b-m+2}).
\end{align*}
\end{lemma}

\begin{proof} We argue by induction on $m$.  First suppose $m=2$. Then 
$$s(a)s(1^b)=s(a+1,1^{b-1})+s(a,1^b)$$

 by Pieri's Formula. Both $(a+1,1^{b-1})$ and $(a,1^b)$ are $2$-adapted, so the result holds in this case.
 
 \q Now suppose that $m\geq 3$ and the result holds for $m-1$. By Lemma 5.3 and the induction hypothesis we have 
 \begin{align*}C_m((s(a)&s(1^b)s(1^{m-3})\ldots s(1))s(1^{m-2}))\cr
 &=C_m(C_{m-1}(s(a)s(1^b)s(1^{m-3})\ldots s(1)) s(1^{m-2}))\cr
 &=C_m(s(a+1,m-2,\ldots,2,1^{b-m+2})s(1^{m-2}))\cr
 &\phantom{hello}+C_m(s(a,m-2,\ldots,2,1^{b-m+3})s(1^{m-2})).
\end{align*}

\q But now, again by Pieri's Formula, for $\lambda=(a+1,m-2,\ldots,2,1^{b-m+2})$ we have 
$$s(\lambda)s(1^{m-2})=\sum_\mu s(\mu)$$
where the sum is over all partitions whose diagram is obtained by adding a box to $m-2$ rows of the Young diagram of $\lambda$.  But such a diagram is not $m$-adapted unless the boxes are added to rows $2$ up to $m-1$ and in this case we have $\mu=(a+1,m-1,\ldots,2,1^{b-m+1})$.  Hence we obtain 
$$C_m(s(a+1,m-2,\ldots,2,1^{b-m+2})s(1^{m-2}))=s(a+1,m-1,\ldots,2,1^{b-m+1})$$ 
and similarly
$$C_m(s(a,m-2,\ldots,2,1^{b-m+3})s(1^{m-2}))=s(a,m-1,\ldots,2,1^{b-m+2}),$$
so we are done.

\end{proof}

\q  We write  $C^*_m$ for $C^*_{\sigma_m}$.   We note that for a symmetric polynomial $g$  contributions to $C^*_m(g)$ can only come from Schur polynomials $s(\lambda)$ with $\lambda$ an $m$-adapted partition. So  we have $C^*_m(g)=C^*_m(C_m(g))$, for a symmetric polynomial $g$. Using Lemma 5.4 it is now easy to verify the following.

\begin{corollary}

Let $m\geq 2$, $a\geq m$ and $b\geq m-1$. Then we have, for all $n$ sufficiently large,  
\begin{align*}&C_m^*(s(a)s(1^b)s(1^{m-2})\ldots s(1))\cr&
=\begin{cases} s(a+1,m-1,\ldots,2,1^{b-m+1}), & \hbox{ if }     a-m    \hbox{ is odd and} \     b-m \hbox{ is  even};\cr
s(a,m-1,\ldots,2,1^{b-m+2}), &  \hbox{ if }     a-m    \hbox{ is even }     b-m \hbox{ is  odd};\cr
0, & \hbox{otherwise.}
\end{cases}
\end{align*}
\end{corollary}

\q Now the  module 
$$S^aE\otimes \tbw^bE\otimes  \tbw^{m-2}E\otimes \cdots \otimes \tbw^2E\otimes E$$
 has a good filtration, e.g. by \cite[Section 4, (3)]{DStd}. Interpreting Corollary 5.5 in terms of $G(n)$-modules we obtain the following.

\begin{corollary} Assume that  $m\geq 2$. Let $a\geq m$ and $b\geq m-1$.  For all $n$ sufficiently large, the  component of the module  
$$S^aE\otimes \tbw^bE\otimes   \tbw^{m-2}E\otimes \cdots \otimes \tbw^2E\otimes E$$
 corresponding to the core $\sigma_m$ is:
\begin{align*}   
&\nabla(a+1,m-1,\ldots,2,1^{b-m+1}),  \q  \hbox{ if }     a-m    \hbox{ is odd and} \     b-m \hbox{ is  even};\cr
&\nabla(a,m-1,\ldots,2,1^{b-m+2}),  \q\q \   \hbox{ if }     a-m    \hbox{ is even }     b-m \hbox{ is  odd};\cr
&0,  \hskip 150pt   \hbox{otherwise.}
\end{align*}

\end{corollary}

We specialise to the following.

\begin{corollary} Assume that $m\geq 2$. Let $a\geq m$ and $b\geq m-1$ with $a-m$ even and $b-m$ odd. Then, for all $n$ sufficiently large,  the component of the module $$S^aE\otimes \tbw^bE\otimes \tbw^{m-2}E\otimes \cdots \otimes \tbw^2E\otimes E$$
corresponding to the core $\sigma_m$ is
$$\nabla(a,m-1,\ldots,2,1^{b-m+2}).$$
\end{corollary}

\section{Decomposable Specht modules}

\q We continue to assume that $q=-1$ and so $l=2$. Let $n,r\geq 0$ with $n\geq r$.  Recall that for  $\alpha\in\Lambda(n,r)$ we write $H(\alpha)$ for the subalgebra \\ $H(\alpha_1)\otimes\cdots\otimes H(\alpha_n)$ of $H(r)$ and that we have 

$$f(S^\alpha E)=H(r)\otimes_{H(\alpha)}K$$

$$f(\Lambda^\alpha E)=H(r)\otimes_{H(\alpha)}K_{\sgn}$$

 \q For finite strings of non-negative integers $\alpha=(\alpha_1,\dots,\alpha_a)$ and $\beta=(\beta_1,\ldots,\beta_b)$ we write $(\alpha\vert \beta)$ for the concatenation  $(\alpha_1,\ldots,\alpha_a,\beta_1,\ldots,\beta_b)$. Assume that $a+b\leq n$ and that $\deg(\alpha)=r_1$,  $\deg(\beta)=r_2$ with $r=r_1+r_2$. Then it follows that 
$$f(S^\alpha E\otimes \tbw^\beta E)=H(r)\otimes_{H(\alpha\vert\beta)} (K\otimes K_\sgn).$$

  But since $q=-1$ we have that $K_\sgn\cong K$, e.g. by \cite[Lemma 3.1]{DJ2}, and so 

$$f(S^\alpha E\otimes \tbw^\beta E)=H(r)\otimes_{H(\alpha\vert\beta)} K=M(\alpha\vert \beta).$$

\smallskip

\q Applying the Schur functor to Corollary 5.7 we get  immediately the following result.

\begin{corollary} Assume that $m\geq 2$. Let $a\geq m$ and $b\geq m-1$ with $a-m$ even and $b-m$ odd.  Then the block component  of the module $M(a,b,m-2,\ldots,2,1)$ corresponding to the core $\sigma_m$ is
$$\Sp(a,m-1,\ldots,2,1^{b-m+2}).$$
\end{corollary}

\q Comparing now Corollary 4.2 with Corollary 6.1 we get the main result of this paper.

\begin{theorem} Assume that $m\geq 2$. Let $a\geq m$ and $b\geq m-1$ with $a-m$ even and $b-m$ odd.  Then we have 
$$\Sp(a,m-1,\ldots,2,1^{b-m+2})=\bigoplus_{\mu} Y(\sigma_m+2\mu)$$
where the sum is over all partitions $\mu=(c,d)$ such that $c+d=u+v$ and $(c-d,u-v)$ is $p$-special, where  $a-m=2u$ and $b-m+1=2v$. 

\end{theorem}

\bs 

\q We give now an example of such a decomposition to point out that the number of indecomposable summands is unbounded.

\begin{example}

 Assume that $K$ has characteristic $2$. By Theorem 6.2 and using a simple inductive argument it follows that for $k\geq 1$ we have 

$$\Sp(2^k+2,1^{2^k-1})=\bigoplus_{j=1}^k Y(2^k+2^j,2^k-2^j+1).$$

\end{example}

\bs

\section{Further Remarks}

\subsection{Hook partitions}

\q As an immediate application of Theorem 6.2 for $m=2$, one obtains a decomposition for the Specht module $\Sp(a,1^b)$ where $a$ is even and $b$ is odd. We point out here that our method gives also a decomposition of $\Sp(a,1^b)$ when $a$ is odd and $b$ is even. In fact, we have the following proposition.

\begin{prop}

Let $a,b\geq1 $ and assume that $a$ and $b$ have different parity. Then we have the following decompositions

(i) For $a$ even and $b$ odd,

$$\Sp(a,1^b)= \bigoplus_{\mu} Y(\sigma_2+2\mu),$$

where the sum is over all partitions $\mu=(c,d)$, such that $c+d=u+v$ and $(c-d,u-v)$ is $p$-special, where  $a-2=2u$ and $b-1=2v$.

(ii) For $a$ odd and $b$ even,

$$\Sp(a,1^b)= \bigoplus_{\mu} Y(\sigma_1+2\mu),$$

where the sum is over all partitions $\mu=(c,d)$, such that $c+d=u+v$ and $(c-d,u-v)$ is $p$-special, where  $a-1=2u$ and $b=2v$.

\end{prop}

\begin{proof}

 Part (i) follows directly from Theorem 6.2 for $m=2$. So we can move on to the second part. Let $a$ be odd and $b$ be even. 

\q We consider the tensor product $S^aE\otimes \tbw^bE$. This module decomposes as

$$S^aE\otimes \tbw^bE=\nabla(a+1,1^{b-1})\oplus \nabla(a,1^b), \eqno{(\dagger)}$$

where the partition $(a+1,1^{b-1})$ has $2$-core $\sigma_2$ and the partition $(a,1^b)$ has $2$-core $\sigma_1$.

\q Now, after applying the Schur functor to $(\dagger)$ we get, as in Section 6, that 

$$M(a,b)=\Sp(a+1,1^{b-1})\oplus \Sp(a,1^b).$$

Hence, the block component of $M(a,b)$ corresponding to the core $\sigma_1$ is $\Sp(a,1^b)$. 

\q Moreover it is easy to see, as in Corollary 4.2, that the block component of $M(a,b)$ corresponding to the core $\sigma_1$ is 

$$\bigoplus_{\mu} Y(\sigma_1+2\mu),$$

where the sum is over all partitions $\mu=(c,d)$, such that $c+d=u+v$ and $(c-d,u-v)$ is $p$-special, where  $a-1=2u$ and $b=2v$. Comparing now these two expressions we get immediately the result.

\end{proof}

\subsection{More Decomposable Specht Modules}

\q We point out here that the decomposition of the Specht modules described in Theorem 6.2, lays the foundations for the discovery of other families of decomposable Specht modules. We describe such a case here and we take our considerations further in \cite{DGMoreDeco}. For simplicity throughout this subsection we assume that $K$ has characteristic $2$. The Hecke algebras analogues of the results of this section will appear in a much more general form in \cite{DGMoreDeco}. 

\q Since $K$ has characteristic $2$ there is no need from now on to distinguish between the irreducible modules $L(\lambda)$ and $\dotL(\lambda)$.  We will need the following lemma regarding the dimension of some weight spaces for certain irreducible $G(n)$-modules.

\begin{lem}

Let $a,b\geq2$ with $a$ be even and $a\not\equiv 0\ \mod\ 4$, and $b$ be odd. Let $\mu=(\mu_1,\mu_2)$ be a two part partition  with expansion $\mu=\mu^0+2\bar{\mu}$. Then 

\begin{equation*}
\dim L(\sigma_2+2\mu)^{(a,b,2)}=
\begin{dcases*}
2\dim L(\mu)^{(u+1,v)}+\dim L(2\bar{\mu})^{(u-2,v)}, & if  $\mu^0=(2,1);$\cr
2\dim L(\mu)^{(u+1,v)}, & otherwise
\end{dcases*}
\end{equation*}

where $a=2(u+1)$ and $b=2v+1$.

\end{lem}

\begin{proof}
 
 The weights for the $G(3)$-module $L(\sigma_2)$ are
 
 $$\{(2,1,0),(2,0,1),(0,2,1),(0,1,2),(1,0,2),(1,2,0),(1,1,1)\},$$
 
 e.g. by considering the Schur function $s(2,1,0)$. Since $a$ is even and $b$ is odd the only weights in this list that can contribute to the $(a,b,2)$ weight space of  $L(\sigma_2+2\mu)=L(\sigma_2)\otimes L(\mu)^F$ are $(2,1,0)$ and $(0,1,2)$. In these cases we can have that $(a,b,2)=(2,1)+2\rho$ and $(a,b,2)=(0,1,2)+2\xi$ for some  weights $\rho$ and $\xi$ of $L(\mu)$. Therefore we get that 
 
$$\dim L(\sigma_2+2\mu)^{(a,b,2)}=\dim L(2\mu)^{(a-2,b-1,2)}+\dim L(2\mu)^{(a,b-1)}.$$

\q First notice that if $\mu^0=0$ then by Steinberg's tensor product theorem we get directly that $L(2\mu)^{(a-2,b-1,2)}=0$. Moreover since $a\not\equiv 0\ \mod\ 4$ we also have that $L(2\mu)^{(a,b-1)}=0$ and so the result holds in this case. 

\q Hence, we might assume that $\mu^0\neq 0$ and so it has one of the forms $(1),(1,1)$ or $(2,1)$.

\q  We write $a$ and $b$ in the form $a=2(u+1)$ and $b=2v+1$. Since $a\not\equiv 0\ \mod\ 4$, then $u$ must be even. We have that

 $$\dim L(2\mu)^{(a-2,b-1,2)}=\dim L(\mu)^{(u,v,1)}\ \ \ {\text{and}}$$ 
 
 $$\dim L(2\mu)^{(a,b-1)}=\dim L(\mu)^{(u+1,v)}.$$ 

\q We consider now all the possible remaining cases for $\mu^0$. Using Steinberg's tensor product theorem one easily obtains the following.

i) Let $\mu^0=(1)$.

Then  $\dim L(\mu)^{(u,v,1)}=\dim L(2\bar{\mu})^{(u,v)}$. 

Also $\dim L(\mu)^{(u+1,v)}=\dim L(2\bar{\mu})^{(u,v)}$, since $u$ is even, and the result holds.

ii) Let $\mu^0=(1,1)$.

As above we have that $\dim L(\mu)^{(u,v,1)}=\dim L(2\bar{\mu})^{(u,v-1)}$, since $u$ is even.

Moreover, $\dim L(\mu)^{(u+1,v)}=\dim L(2\bar{\mu})^{(u,v-1)}$ and so we get again the result.

iii) Finally let $\mu^0=(2,1)$. 

Here we get that $\dim L(\mu)^{(u,v,1)}=\dim L(2\bar{\mu})^{(u-2,v)}+\dim L(2\bar{\mu})^{(u,v-2)}$. 

On the other hand we have that $\dim L(\mu)^{(u+1,v)}=\dim L(2\bar{\mu})^{(u,v-2)}$.

 This completes all the cases and proves the result.

\end{proof}

\begin{prop}

Let $a\geq4 , b\geq 3$ with $a$ even and $a\not\equiv 0\ \mod\ 4$, and let $b$ be odd. Then

$$\Sp(a,3,1^{b-1})=\bigoplus_{\mu} Y(\sigma_2+2\mu)\ \oplus \ \bigoplus_{\rho}Y(\sigma_2+2\rho),$$

where the sums are over all partitions $\mu$ and $\rho$ such that,

 $\mu=(\mu_1,\mu_2,1)$, $\mu_1+\mu_2=u+v$ and $(\mu_1-\mu_2,u-v)$ is $2$-special; and

$\rho=(\rho_1,\rho_2)$ with $\rho=\sigma_2+2\bar{\rho}$, $2(\bar{\rho}_1+\bar{\rho}_2) =u+v-2$ and $(2(\bar{\rho}_1-\bar{\rho}_2),u-v-2)$ is $2$-special, where  $a=2(u+1)$ and $b=2v+1$ for some $u,v\geq 1$.
\end{prop} 

\begin{proof}

We consider the $G(n)$-module $S^aE\otimes \tbw ^bE\otimes S^2E$, for $n$ sufficiently large. This has the following decomposition 

$$S^aE\otimes \tbw^bE\otimes S^2E=(\nabla(a+1,1^{b-1})\otimes S^2E)\oplus (\nabla(a,1^b)\otimes S^2E),$$

where the partition $(a+1,1^{b-1})$ has $2$-core $\sigma_1$ and $(a,1^b)$ has $2$-core $\sigma_2$.

\q Now using Pieri's Formula we can easily see that the block component of  $\nabla(a+1,1^{b-1})\otimes S^2E$ corresponding to the core $\sigma_2$ is just $\nabla(a+2,1^b)$.

\q Let $V$ be the block component of the module $\nabla(a,1^b)\otimes S^2E$ corresponding to the core $\sigma_2$. Then $V$ has a good filtration and again by Pieri's Formula we see that $V$ actually fits in the short exact sequence

$$0\rightarrow \nabla(a,3,1^{b-1})\rightarrow V\rightarrow \nabla(a+2,1^b)\rightarrow 0.$$

By \cite[Section 4.2 (17)]{q_Schur}, one has that

$$\Ext^1_{G(n)}( \nabla(a+2,1^b),\nabla(a,3,1^{b-1}))=\Ext^1_{G(2)}( \nabla(a+1),\nabla(a-1,2)),$$

 and since $a\not\equiv 0\ \mod\ 4$, we have that $\Ext^1_{G(2)}( \nabla(a+1),\nabla(a-1,2))=0$, see for e.g. \cite[Corollary 5.12]{DGExtensions}. Therefore, $V=\nabla(a,3,1^{b-1})\oplus  \nabla(a+2,1^b)$. 

\q Hence, the block component of  $S^aE\otimes \tbw ^bE\otimes S^2E$ corresponding to the core $\sigma_2$ is 

$$\nabla(a,3,1^{b-1})\oplus \nabla(a+2,1^b)\oplus \nabla(a+2,1^b).$$

\q Now as in Sections 4 and 6, when we apply the Schur functor to the $G(n)$-module $S^aE\otimes \tbw ^bE\otimes S^2E$  and restrict attention to the block component of $M(a,b,2)$ corresponding to the core $\sigma_2$, we get 

$$\Sp(a,3,1^{b-1})\oplus \Sp(a+2,1^b)\oplus \Sp(a+2,1^b)=\bigoplus_{\mu} Y(\sigma_2+2\mu)^{(d_\mu)},$$

where the sum is over all partitions $\mu=(\mu_1,\mu_2,\mu_3)$ with

 $d_\mu=\dim L(\sigma_2+2\mu)^{(a,b,2)}\neq 0$.

\q Assume first that $\mu_3\neq0$. If $L(\sigma_2+2\mu)^{(a,b,2)}\neq 0$, then $(a,b,2)\leq \sigma_2+2\mu$ and so we must have $\mu_3=1$.

Therefore in this case 

\begin{align*}
\dim L(\sigma_2+2\mu)^{(a,b,2)}&=\dim L(\sigma_2+2(\mu-\omega_3))^{(a-2,b-2)}\cr
&=\dim L(2(\mu-\omega_3))^{(a-4,b-3)}.
\end{align*}

 We write $a=2(u+1)$ and $b=2v+1$. Then 
 
 $\dim L(2(\mu-\omega_3))^{(a-4,b-3)}=\dim L(\mu_1-1,\mu_2-1)^{(u-1,v-1)}$, and by Lemma 3.3 this dimension is $1$ if $(\mu_1-\mu_2,u-v)$ is $2$-special and $0$ otherwise.

\q Assume now that $\mu_3=0$ and so $\mu=(\mu_1,\mu_2)$. Then Lemma 7.2.1 and Lemma 3.3 give the dimension of $L(\sigma_2+2(\mu_1,\mu_2))^{(a,b,2)}$. 

\q Putting these two points together we conclude that the direct sum 

 $\Sp(a,3,1^{b-1})\oplus \Sp(a+2,1^b)\oplus \Sp(a+2,1^b)$, decomposes as

$$\bigoplus_{\mu} Y(\sigma_2+2\mu)\ \oplus\ \bigoplus_{\nu}Y(\sigma_2+2\nu)^{(2)} \ \oplus \ \bigoplus_{\rho}Y(\sigma_2+2\rho),  \eqno{(1)}$$

where the sums are over all partitions $\mu,\nu$ and $\rho$ respectively such that,

$\mu=(\mu_1,\mu_2,1)$, $\mu_1+\mu_2=u+v$ and $(\mu_1-\mu_2,u-v)$ is $2$-special;

$\nu=(\nu_1,\nu_2)$, $\nu_1+\nu_2=u+v+1$ and $(\nu_1-\nu_2,u-v+1)$ is $2$-special; 

$\rho=(\rho_1,\rho_2)$ with $\rho=\sigma_2+2\bar{\rho}$, $2(\bar{\rho}_1+\bar{\rho}_2) =u+v-2$ and $(2(\bar{\rho}_1-\bar{\rho}_2),u-v-2)$ is $2$-special.

\q On the other hand, by Theorem 6.2 we have 

$$\Sp(a+2,1^b)=\bigoplus_{\nu} Y(\sigma_2+2\nu), \eqno{(2)}$$
where the sum is over all partitions $\nu=(\nu_1,\nu_2)$ such that $\nu_1+\nu_2=u+v+1$ and $(\nu_1-\nu_2, u-v+1)$ is $2$-special.

\q Comparing now the decompositions $(1)$ and $(2)$ we get immediately that 

$$\Sp(a,3,1^{b-1})=\bigoplus_{\mu}Y(\sigma_2+2\mu)\ \oplus \ \bigoplus_{\rho}Y(\sigma_2+2\rho),$$

where the sums run over all partitions $\mu$ and $\rho$ as described in the statement.
\end{proof}

\smallskip

\begin{rmk}
{\rm{Proposition 7.2.2 gives many new decomposable Specht modules of the form $\Sp(a,3,1^b)$ which do not appear in the list of Dodge and Fayers, \cite[Theorem 3.1 and Corollary 3.2]{DF}. For instance we have the following example.}}

\end{rmk}

\begin{exam}

For $a=14$ and $b=9$ we have that 

$$\Sp(14,3,1^8)=Y(14,9,2)\oplus Y(18,5,2)\oplus Y(14,11).$$

\end{exam}

\medskip

\q Of course we can now obtain new decomposable Specht modules by considering the linear duals of the modules appeared above. More precisely, by \cite[Theorem 8.15]{James} we have that in characteristic $2$, $\Sp(\lambda)^*=\Sp(\lambda')$, where $\Sp(\lambda)^*$ is the linear dual of $\Sp(\lambda)$ and $\lambda'$ is the transpose of the partition $\lambda$. Moreover, we have that the Young modules are self-dual modules.

\q Therefore by considering  the linear duals of the modules appeared in Proposition 7.2.2 we get the following result.

\begin{corol}

Let $a\geq4 , b\geq 3$ with $a$ even and $a\not\equiv 0\ \mod\ 4$, and $b$ odd. Then

$$\Sp(b+1,2,2,1^{a-3})=\bigoplus_{\mu} Y(\sigma_2+2\mu)\ \oplus \ \bigoplus_{\rho}Y(\sigma_2+2\rho),$$

where the sums are over all partitions $\mu$ and $\rho$ such that,

 $\mu=(\mu_1,\mu_2,1)$, $\mu_1+\mu_2=u+v$ and $(\mu_1-\mu_2, u-v)$ is $2$-special; and

$\rho=(\rho_1,\rho_2)$ with $\rho=\sigma_2+2\bar{\rho}$, $2(\bar{\rho}_1+\bar{\rho}_2) =u+v-2$ and $(2(\bar{\rho}_1-\bar{\rho}_2), u-v-2)$ is $2$-special, where  $a=2(u+1)$ and $b=2v+1$ for some $u,v\geq 1$.

\end{corol}

\section*{Acknowledgements}

The second author gratefully acknowledges the support of The Royal Society through a University Research Fellowship.

\end{document}